\documentclass[11pt]{article}
\usepackage{amsmath}
\usepackage{amsfonts}
\usepackage{amssymb}
\usepackage{amsthm}
\usepackage{mathrsfs}
\usepackage{lscape}
\usepackage{euscript}
\usepackage{latexsym,enumerate,color,mdwlist}
\usepackage{a4wide}
\usepackage{hyperref}
\usepackage{soul}

\usepackage{mathtools}

\usepackage[shortlabels]{enumitem}
\setlist[enumerate]{leftmargin=25pt}
\setlist[itemize]{leftmargin=25pt}

\allowdisplaybreaks

\theoremstyle{plain} \newtheorem{Thm}{Theorem}[section]
\newtheorem{prop}[Thm]{Proposition}
\newtheorem{lem}[Thm]{Lemma}

\newtheorem{cor}[Thm]{Corollary}

\theoremstyle{definition}
\newtheorem{de}[Thm]{Definition}
\newtheorem{obs}[Thm]{Observation}
\newtheorem{rem}[Thm]{Remark}
\newtheorem{clm}{Claim}
\newtheorem{que}{Question}
\newtheorem{con}{Conjecture}

\newtheorem{summary}{Reduction}
\newtheorem*{theorem*}{Theorem}
\theoremstyle{remark}

\newcommand{\tG}{\widetilde{G}}

\newcommand{\rad}{{\sf Rad}}

\newcommand{\wh}[1]{\widehat{#1}}

\newcommand{\ord}{{\rm ord}}

\newcommand{\comment}[1]{}

\def\N{\mathbb{N}}
\def\Q{\mathbb{Q}}
\def\Z{\mathbb{Z}}
\def\C{\mathbb{C}}
\def\R{\mathbb{R}}
\def\La{\Lambda}
\def\la{\lambda}
\def\lala{\La-\La}
\def\stb{,\ldots ,}

\def\stb{,\ldots ,}

\newcommand{\set}[1]{{\left\{{#1}\right\}}}

\DeclarePairedDelimiter\abs{\lvert}{\rvert}

\newcommand{\vn}{\ensuremath{\varnothing}}
\newcommand{\ssq}{\ensuremath{\subseteq}}

\newcommand{\Php}{\ensuremath{\Phi_p(x^{N/p})}}
\newcommand{\Phq}{\ensuremath{\Phi_q(x^{N/q})}}

\newcommand{\pdiv}{\ensuremath{\mid\!\mid}}

\newcommand{\ol}{\ensuremath{\overline}}

\newcommand{\out}{}

\begin{document}

\title{On the discrete Fuglede and Pompeiu problems}
\date{}
\author{Gergely Kiss \thanks{University of Luxembourg, Faculty of Science, Mathematics Research Unit, e-mail: kigergo57@gmail.com}
\qquad Romanos Diogenes Malikiosis \thanks{Aristotle University of
Thessaloniki, Department of Mathematics. e-mail: rwmanos@gmail.com}
\qquad G\'abor Somlai \thanks{E\"otv\"os Lor\'and University,
Faculty of Science, Institute of Mathematics and MTA-ELTE Geometric
and Algebraic Combinatorics Research Group, NKFIH 115799.
    e-mail: gsomlai@cs.elte.hu}
\qquad M\'at\'e Vizer \thanks{Alfr\'ed R\'enyi Institute of
Mathematics, Hungarian Academy of Sciences. e-mail:
vizermate@gmail.com.}}

\maketitle

\begin{abstract}

We investigate the discrete Fuglede's conjecture and Pompeiu problem
on finite abelian groups and develop a strong connection between the
two problems. We give a geometric condition under \out{which} a
multiset of a finite abelian group has the discrete Pompeiu
property. Using this description and the revealed connection we
prove that Fuglede's conjecture holds for $\mathbb{Z}_{p^n q^2}$,
where $p$ and $q$ are different primes. In particular, we show that
every spectral subset of $\mathbb{Z}_{p^nq^2}$ tiles the group.
Further, using our combinatorial methods we give a simple proof for
the statement that Fuglede's conjecture holds for $\Z_p^2$.

\end{abstract}

\section{Introduction}

In this article we deal with the discrete version of Fuglede's
conjecture and Pompeiu problem, both originated in analysis. We
build a relationship between them that helps us to provide new
results for Fuglede's conjecture in the discrete setting.

The following question was asked by Pompeiu \cite{P1929}. Take a
continuous function $f$ on the plane whose integral is zero on every
unit disc. Does it follow that $f$ is constant zero? The answer for
this question is no, but it initiated several different type of
investigations in various settings, and in some cases the answer is
affirmative for an analogous question. We give an implicit
characterization of the non-Pompeiu sets for finite abelian groups.

Fuglede conjectured \cite{F1974} that a bounded domain $S \subset
\mathbb{R}^d$ tiles the $d$-dimensional Euclidean space if and only
if the set of $L^2(S)$ functions admits an orthogonal basis of
exponential functions. This conjecture was disproved by Tao
\cite{T2004}.

A discrete version of Fuglede's conjecture might be formulated in
the following way. A subset $S$ of a finite abelian group $G$ tiles
$G$ if and only if the character table of $G$ has a submatrix, whose
rows are indexed by the elements of $S$, which is a complex Hadamard
matrix. This version of Fuglede's conjecture is not only interesting
for its own by also plays a crucial role in the above mentioned
counterexample of Tao. Actually his counterexample (in
$\mathbb{R}^{5}$) is based on a counterexample for elementary
abelian $p$-groups of finite rank.

Fuglede's conjecture is especially interesting for finite cyclic
groups, since e.g. every tiling of $\mathbb{Z}$ is periodic, so it
goes back to a tiling of a finite cyclic group. However, not much is
known for cyclic groups. A recent paper of the second author and
Kolountzakis \cite{MK17} shows that Fuglede's conjecture holds for
any cyclic group of order $p^nq$, where $p$ and $q$ are different
primes.

Our main contribution towards Fuglede's conjecture for cyclic groups
is to connect this problem with the Pompeiu problem, introduce more
combinatorial ideas and verify it for yet unknown cases: cyclic
groups of order $p^nq^2$, $n\geq 1$ (see Theorem \ref{thmfotetel}).

Further using our techniques we give a neat and combinatorial proof
for the previously known fact (proved by Iosevich, Mayeli and
Pakianathan \cite{{IMP2017}}) that Fuglede's conjecture holds for
$\mathbb{Z}_p^2$ (see also Theorem \ref{ThmZp2}).


\

\textbf{Structure of the paper.} The paper is organized as follows.
Section \ref{section2} is devoted to a detailed introduction to
Fuglede's conjecture and the Pompeiu problem, introducing also the
discrete version of them. Further we establish a connection between
the two problems. In section \ref{sectionPompfor abelian} we give
some sort of solution for the Pompeiu problem for
abelian groups that we apply later in Section 
\ref{sectiont2}. Section \ref{section5} and \ref{section6} are
preparations for the proof of Theorem \ref{thmfotetel}. In Section
\ref{section5} we reduce the cases to a special one partly based on
our results concerning the Pompeiu problem. In Section
\ref{section6} we prove some technical lemmas, that we
constantly use later. Section \ref{sectiont2} 
is devoted to the main proof of Theorem \ref{thmfotetel}. Finally,
in the Appendix we give an alternative proof of Theorem
\ref{ThmZp2}.

\section{Fuglede and Pompeiu problem}\label{section2}
\subsection{Fuglede's Spectral Set Conjecture}

The original conjecture of Fuglede \cite{F1974} was formulated as
follows. Let $\Omega$ be a measurable subset of $\mathbb{R}^{n}$ of
positive Lebesgue measure. A set $\Omega\subseteq \R^n$ is called
{\it spectral} if there is a set $\Lambda\subset \R^n$ such that
$\{e^{i\lambda \cdot x}:\lambda \in \Lambda, \ x\in \Omega\}$ is an
orthogonal basis of $L^2(\Omega)$. Then $\Lambda \subseteq
\mathbb{R}^{n}$ is called the \textit{spectrum} of $\Omega$.


We say that $S$ is a \textit{tile} of $\mathbb{R}^{n}$, if there is
a set $T \subset \mathbb{R}^{n}$ such that almost every point of
$\mathbb{R}^{n}$ can be uniquely written as $s + t$, where $s \in S$
and $t \in T$. In this case, we say that $T$ is the \textit{tiling
complement of $S$}.

Fuglede's Spectral Set Conjecture (that we just call Fuglede's
conjecture) \cite{F1974} states the following:

\begin{con}
$\Omega$ is spectral if and only if $\Omega$ is a tile.
\end{con}

The conjecture was proved by Fuglede \cite{F1974} in the special
case, when the tiling complement or the spectrum is a lattice in
$\R^n$. Also it has been verified by Fuglede that the $L^2$-space
over a triangle or a disc does not admit an orthogonal basis of
exponentials. (The proof for the disc was corrected by Iosevich,
Katz
and Pedersen \cite{IKP99}.) The conjecture was further verified in some 
other cases (see e.g. \cite{IKT2003,L2001}).

Tao \cite{T2004} disproved the spectral-tile direction of the
conjecture by constructing a spectral set in $\R^5$ that does not
tile the $5$ dimensional space. As an extension of Tao's work,
Matolcsi \cite{M2004} proved that (the same direction of) the
conjecture fails in dimension $4$ as well. Further, Kolountzakis and
Matolcsi \cite{KM2006,KM2006a} and Farkas, Matolcsi and M\'ora
\cite{FMM2006}
provided counterexamples in dimension $3$ for each 
direction of the
conjecture.

\

\textbf{Discrete abelian groups}. Fuglede's conjecture can be
naturally stated for
other groups, for example $\mathbb{Z}$.  
These cases are not only interesting on their own, but they also
have connection with the original case, since e.g.~in his disproof
of the 5-dimensional case, Tao constructed a spectral set in
$\mathbb{Z}^{5}_{3}$ (containing $6$ elements, hence not a tile, as
the cardinality of any tile of a finite abelian group divides the
order of the group), then he lifted this counterexample to
$\mathbb{R}^{5}$. Similar strategy was carried out by Kolountzakis
and Matolcsi in the disproof of the other direction of the original
conjecture, see \cite{KM2006a}. We also mention some examples, where
Fuglede's conjecture holds. These include finite cyclic $p$-groups
\cite{L2002}, $\mathbb{Z}_p \times \mathbb{Z}_p$ \cite{IMP2017}, and
$\mathbb{Q}_p$ \cite{FFLS2015}, the field of $p$-adic numbers.

Borrowing the notation from \cite{DL2014} and \cite{MK17}, we write
$\mathbf{S-T}(G)$ (resp. $\mathbf{T-S}(G)$), if the $Spectral
\Rightarrow Tile$ (resp. $Tile \Rightarrow Spectral$) direction of
Fuglede's conjecture holds in $G$ for every bounded subset. The
above mentioned connection between the conjecture on $\mathbb{R}$,
on $\mathbb{Z}$ and on finite cyclic groups is summarized below
\cite{DL2014} (where $\mathbf{T-S}(\mathbb{Z}_{\mathbb{N}})$ means
that $\mathbf{T-S}(\mathbb{Z}_n)$ holds for every $n \in \N$):

\
$$\mathbf{T-S}(\mathbb{R}) \Leftrightarrow \mathbf{T-S}(\mathbb{Z})
\Leftrightarrow \mathbf{T-S}(\mathbb{Z}_{\mathbb{N}}),$$

$$\mathbf{S-T}(\mathbb{R}) \Rightarrow \mathbf{S-T}(\mathbb{Z})
\Rightarrow \mathbf{S-T}(\mathbb{Z}_{\mathbb{N}}).$$ \

According to this, a counterexample to the $Spectral \Rightarrow
Tile$ direction in a finite cyclic group can be lifted to a
counterexample in $\mathbb{R}$; on the other hand, if the same
direction of the conjecture were true for every cyclic group or even
in $\mathbb{Z}$, this would hold no meaning for the original
conjecture in $\mathbb{R}$.

    Concerning tiles in discrete groups it was proved in \cite{Newman1997} that if $S$ is a finite set,
which tiles $\Z$ with tiling complement $T$, then $T$ is periodic
i.e. $T+N=T$ for some $N \in \Z$. This shows that every tiling of
the integers reduces to a tile of a cyclic group $\Z_N$ for some $N
\in \mathbb{N}$.

We also mention a related result of R\'edei \cite{Redei65}. We say
that $A_1+ \dotsb +A_k$ is a {\it factorization} of the abelian
group $G$ if every element of $G$ can uniquely be written as the sum
of one element from each $A_i$.

\begin{Thm}(\cite{Redei65})\label{redei}
Let $G = A_1 +A_2+ \dotsb + A_n$ be a factorization of an abelian
group $G$, where each $A_i$ contains $0$ and is of prime
cardinality. Then at least one of the sets $A_i$ is a subgroup of
$G$.
\end{Thm}

\

\textbf{Cyclic groups.} Surprisingly, despite their previously
described role in the discrete version of Fuglede's conjecture, not
much is known for cyclic groups. A recent result of the second
author and Kolountzakis \cite{MK17} proved Conjecture
\ref{fugconcyc} (see later) for $\mathbb{Z}_{p^n q}$. They also
wrote that most probably, their result might be extended to cyclic
groups of order having $2$ different prime divisors but they haven't
succeeded yet.

\

As we will mainly deal with cyclic groups, let us state the
conjecture again in this setting. First let us define spectral sets
and tiles in cyclic groups also.

\begin{de}\label{dede}
For a set $S \subset \mathbb{Z}_N$, we say that that $S$ is {\it
spectral} if $L^2(S)$ has an orthogonal basis of exponentials
(indexed by $\La$). This is equivalent to the following two
conditions to hold:
\smallskip

1. There is $\La \subset \mathbb{Z}_N$ such that any $f: S
\rightarrow \mathbb{C}$ can be written as the $\C$-linear
combination of exponentials of the form \out{
\[ \xi_N^{\la \cdot x } ~~~ (\la \in \La), \]
where the product $\la \cdot x$ is taken modulo $N$ and
$\xi_N=e^{2\pi i/N}$.}


\smallskip

2. For any two different $\la,\la' \in  \La$ we have: $$\sum_{x \in
S} \xi_N^{(\la-\la')\cdot x }=0$$ (i.e. the representations
$\chi_{\lambda}(x)=\xi_N^{\la \cdot x} $ and
$\chi_{\lambda'}(x)=\xi_N^{\la' \cdot x }$ are orthogonal).
\end{de}
We denote $ \{\chi_{\la} \mid \la \in \La\}$ by $\chi_{\Lambda}$.
\begin{rem}
We note that if $S$ is a spectral set, then $\abs{S}=\abs{\La}$
follows from Definition \ref{dede}. \out{Condition 2 further implies
that
\begin{equation}\label{Ladiff}
 \lala\subseteq \set{0} \cup \set{x\in\Z_N:\wh{1}_S(x)=0},
\end{equation}
where $1_S$ is the characteristic function of $S$, and
$\wh{f}(x)=\sum_{y\in \Z_N}f(y)\xi_N^{-x\cdot y}$ is the discrete
Fourier transform of $f:G\mapsto\C$, as usual}.
\end{rem}
\begin{de}
Let $G$ be a discrete abelian group. We say that $S \subset G$  {\it
tiles} $G$ if there exists $T \subset G$ such that $S+T= G$, where
$S+T$ is the set of elements of $G$ of the form $s+t$ ($s \in S, ~t
\in T$), counted with multiplicity, so we have each $g \in G$
exactly once. In this case we say that $T$ is a {\it tiling
complement} of $S$ in $G$.
\end{de}

For cyclic groups Fuglede's conjecture can be stated as follows.
\begin{con}\label{fugconcyc} For any $N$ and $S \subset \mathbb{Z}_N$ we have that $S$ is spectral if and only if $S$ tiles $\Z_N$.
\end{con}

There has been some recent progress on this conjecture over the last few years. The known results for the \emph{Tiling$\Rightarrow$Spectral} direction follow from the
work of Coven \& Meyerowitz \cite{CM1999} and \L{}aba \cite{L2002}. Coven \& Meyerowitz proved that if a finite subset $A$ of the integers satisfies two conditions \ref{t1} 
and \ref{t2} (to be defined below) then it tiles $\Z$ by translations. The inverse holds when the cardinality of $A$ is divisible by at most two primes (Corollary to
\cite[Theorem B2]{CM1999}). 
\L{}aba then connected these properties with Fuglede's conjecture on $\Z$, proving that if $A$ satisfies \ref{t1} and \ref{t2}, then it has a spectrum 
\cite[Theorem 1.5(i)]{L2002}, therefore, if $A$ tiles $\Z$ and its cardinality is divisible by at most two distinct primes, then it must be spectral 
\cite[Corollary 1.6(i)]{L2002}. The passage to cyclic groups of order $p^nq^m$ is easily done through \cite[Lemma 2.3]{CM1999}, which implies that if $A$ tiles $\Z$ and $\abs{A}=p^aq^b$,
then there is a (possibly different) tiling of $\Z$ by translates of $A$ with period $N=p^nq^m$ for some $m,n\in\N$. In other words, $A$ could be considered as a subset of $\Z_N$, and 
the \emph{Tiling$\Rightarrow$Spectral} direction in $\Z_N$ follows verbatim using the above results. 
For a proof containing all the above arguments strictly in the setting of cyclic groups, we refer the reader to \cite[Section 3]{MK17}.

Concerning the case for $N$ square free, the \emph{Tiling$\Rightarrow$Spectral} direction in $\Z_N$ follows easily from the fact that any tile of $\Z_N$ is a set of
coset representatives of a subgroup of $\Z_N$, again from combined arguments from \cite{CM1999} and \cite{L2002}. This fact was posed as a problem in 
Tao's blog\footnote{\url{https://terrytao.wordpress.com/2011/11/19/some-notes-on-the-coven-meyerowitz-conjecture/}}, which was subsequently solved by
\L{}aba\footnote{\url{https://terrytao.wordpress.com/2011/11/19/some-notes-on-the-coven-meyerowitz-conjecture/\#comment-121464}}
and Meyerowitz\footnote{\url{https://terrytao.wordpress.com/2011/11/19/some-notes-on-the-coven-meyerowitz-conjecture/\#comment-112975}}. Their arguments were based 
on \cite[Lemma 2.3]{CM1999}, which implies that a tile $A$ in $\Z_N$ with $N$ square free, accepts the subgroup $M\Z_N$ as a tiling complement, where $M=\abs{A}$. This is
one instance where the properties \ref{t1} and \ref{t2} hold trivially, thus $A$ is also spectral due to \cite[Theorem 1.5(i)]{L2002}. For a self-contained proof 
of the \emph{Tiling$\Rightarrow$Spectral} direction in the setting of cyclic groups of square free order, which is along the same lines, we refer the reader to \cite{Shi18}.

The reverse direction, \emph{Spectral$\Rightarrow$Tiling}, is considerably harder, and the best results so far are the proofs for 
$N=p^nq$ \cite{MK17} and $N=pqr$ \cite{Shi18}, where $p$, $q$, $r$ distinct
primes. The main tool that is introduced in the \emph{Spectral$\Rightarrow$Tiling} direction is the structure of the vanishing sums of roots of unity \cite{vanishingsum}.

In this paper we verify Conjecture \ref{fugconcyc} for cylic groups
of order $p^nq^2$ by proving the following.
\begin{Thm}\label{thmfotetel}
Let $p$ and $q$ be two different primes. Then we have $\mathbf{S-T}
\ (\mathbb{Z}_{p^nq^2})$, for every $n\geq1$.
\end{Thm}
As stated above, $\mathbf{T-S} \
(\mathbb{Z}_{p^nq^m})$ has already been proven \cite{CM1999,L2002}. Combining this result and Theorem
\ref{thmfotetel} we obtain:

\begin{Thm}Let $p$ and $q$ be two different primes. Then
Fuglede's conjecture holds for $ \Z_{p^nq^2}$, $n\geq 1$.
\end{Thm}

Furthermore, using our method we give in the Appendix a simple proof
of the theorem of Iosevich, Mayeli and Pakianathan \cite{IMP2017},
stating that Fuglede's conjecture holds for $\Z_p^2$.

\subsection{Pompeiu problem}
The problem goes back to the seminal paper of Pompeiu \cite{P1929},
where he asked the following question of integral geometry:
\begin{que} \label{q1}
Let $K$ be a compact set of positive Lebesgue measure. Is it true
that if $f:\mathbb{R}^2\to \C$ is a continuous function that
satisfies
\begin{equation}\label{epo10}
\int_{\sigma(K)} \, f(x,y)d\lambda_xd\lambda_y=0
\end{equation}
for every rigid motion $\sigma$ (here $\lambda$ denotes the Lebesque
measure), then $f$ is identically zero (i.e, $f\equiv 0$)?
\end{que}

If $K$ is the closed disc of radius $r>0$, then the answer is
negative.  It was shown by Chakalov \cite{C1944} (see also
\cite{G1988}) that \eqref{epo10} holds if $f(x,y)=\sin(a(x+iy))$
where $a>0$ and $J_1(ra)=0$ ($J_{\lambda}$ denotes the Bessel
function of order $\lambda$).
On the other hand, for every nonempty \out{polygon} (moreover, for
any convex domain with at least one corner) the answer for Question
\ref{q1} is affirmative by the result of Brown, Taylor and Schreiber
\cite{BST1973}. Recently, Ramm \cite{Ram2017} showed that there
exists a $f\not\equiv 0$ function that satisfies the 3-dimensional
analogue of \eqref{epo10} for a bounded domain $K\subseteq \R^3$
with $\mathcal{C}^1$-smooth
boundary if and only if $K$ is a closed ball. 
Extensive literature is concerned with the Pompeiu problem. For the
history of the problem see \cite{R1997} and the bibliographical
survey \cite{Z1992}.


In our paper we investigate the discrete version of the
\textit{Pompeiu problem} on finite abelian groups. We note that the
discrete Pompeiu problem for infinite abelian groups was studied in
\cite{KLV2016, Pu2013, Ze1978}.


\

\textbf{The discrete version of Pompeiu problem for an abelian group $G$.} 
In the sequel we denote the binary operation acting on an abelian
group $G$ by $+$ (as the usual addition).




\begin{de}
Let $G$ be an abelian group.
\begin{itemize}
\item Let $S$ be a nonempty finite subset of $G$.
We say that {\it $S$ has the discrete Pompeiu property} (shortly $S$
is {\it Pompeiu}) if, whenever $f\colon G\to \C$ satisfies
\begin{equation}\label{epd1} \sum_{s \in S}f(s+x)=0 \textrm{ for every } x  \in G,\end{equation}
then $f\equiv 0$.

We say that $S$ is a {\it non-Pompeiu set with respect to} $f$ if
$f\not\equiv 0$ and satisfies \eqref{epd1}.

One can define the disrcrete Pompeiu property for multisets
similarly.
\item We call $w: G\to \mathbb{Q}$ a {\it weight function\footnote{Every weight function is a rational constant multiple of a weight
function with integer coefficients. The Pompeiu property is
invariant by a nonzero constant multiple of a Pompeiu weight
function. Thus we may restrict our attention for those weight
functions which take its values in $\Z$.}} defined on $G$. We say
that {\it $w$ is a Pompeiu weight function}  if for any $f:G\to
\mathbb{C}$
\begin{equation}\label{epd2}  \sum_{g \in G}w(g)f(g+x)=0 \textrm{ for every } x  \in G \end{equation}
implies that $f\equiv 0$.

We say that $w$ is a {\it non-Pompeiu weight function with respect
to $f$} if $f\not\equiv 0$ and satisfies \eqref{epd2}.
\end{itemize}
\end{de}

Note that $S$ is a Pompeiu set if and only if its characteristic
function is a Pompeiu weight function.




\begin{rem}
We can extend the previous definition for arbitrary finite group
$(G, \cdot)$ and weight function $w$ as follows.

Let $w:G \to \mathbb{Q}$.
We denote by $Cay(G,w)$ the \textit{Cayley graph} of $G$ with respect to $w$. The vertex set of $Cay(G,w)$ is $G$ and $g$ is connected to $h$ by an edge with weight $w({g^{-1}h})$ for every $g,h \in G$. We denote by $A_w$ the adjacency matrix of $Cay(G,w)$. 
Using the adjacency matrix $A_w$ of $Cay(G,w)$ we may also say $w$
is  a Pompeiu weight function if and only if $A_wf=0$ implies
$f\equiv 0$. The equation $A_wf=0$ implies that if $f\not\equiv 0$,
then $f$ is an eigenvector of $A_w$ with eigenvalue 0. So $w$ is a
Pompeiu weight function if and only if $0$ is not an eigenvalue of
$A_w$. In the finite case this is equivalent to $A_w$ is invertible.

We note that if $G$ is a cyclic group, then $A_w$ is a circulant
matrix.

\end{rem}


The set of irreducible representations of a finite abelian group $G$
will be denoted by $\widetilde{G}$. Every irreducible representation
of an abelian group is one dimensional (a character). Thus
$\widetilde{G}$ is a group which is isomorphic to $G$. Note that
$\widetilde{G}$ is usually called the \textit{dual group} of $G$.

It is well-known \cite{reprsteinberg} that the set of irreducible
representations form an orthogonal basis of $L^2(G)$ with respect to
the natural scalar product $[\psi,\chi]:=\sum_{g \in G}
\psi(g)\overline{\chi(g)}$ for $\psi, \chi \in \widetilde{G}$. Thus
every function $f:G\to \C$ can be uniquely written as
\begin{equation}\label{elin}
f(x)=\sum_{\chi\in \tG} c_{\chi}\chi(x) ~~~~~~ \forall x\in G,
\end{equation}
for some $c_{\chi}\in \C$.

The following proposition can be deduced from \cite{Sz2001}. In
order to make our paper self-contained, we provide the proof.
\begin{prop}\label{pirr}

If $w$ is a non-Pompeiu weight function with respect to a function
$f$, then $w$ is a non-Pompeiu weight function with respect to all
irreducible representations $\chi$ which has nonzero coefficient
$c_{\chi}$ in \eqref{elin}.
\end{prop}

\begin{proof}
 Let $w$ be a non-Pompeiu with respect to a function $f$, then
$\sum_{s \in G}w(s)f(s+x)=0 \textrm{ for every } x  \in G $. Using
\eqref{elin} we get
 $$0=\sum_{s \in S}w(s)\sum_{\chi \in \widetilde{G} } c_{\chi} \chi(s+x)=
 \sum_{\chi  \in \widetilde{G}} c_{\chi} \sum_{s \in S}w(s)\chi(s+x)=\sum_{\chi \in \widetilde{G}}  \big(c_{\chi}   \sum_{s \in S}w(s)\chi(s)\big)\chi(x),$$
 since $\chi$ is a character.
 This statement holds for every $x \in G$ so we can formulate it as follows:
 \[ \sum_{\chi \in \tG} \big(c_{\chi} \sum_{s \in S}w(s)\chi(s)\big) \chi=0.\]
 Since the irreducible representations are linearly independent over $\C$, the previous equation holds if and only if $\sum_{s \in S}w(s)\chi(s)=0$ for all $\chi$ such that $c_{\chi}\ne 0$.
Multiplying with $\chi(x)$ we obtain $\sum_{s \in
S}w(s)\chi(x+s)=0$. Since this holds for every $x \in G$, this means
that $w$ is a non-Pompeiu with respect to such $\chi$.
\end{proof}
We note that a stronger result was proved by Babai \cite{B1979} who
determined the spectrum of a Cayley graphs of abelian groups. The
set of the eigenvalues of $Cay(G,S)$ is $\{\sum_{s \in S}  \chi(s)
\mid \chi \in \tG \}$.

\begin{cor}
If $S$ is a non-Pompeiu set in a finite abelian group, then $S$ is
non-Pompeiu with respect to some irreducible representation of $G$.
\end{cor}
\begin{rem}
Since the characters (irreducible representations) play the role of exponential functions over the abelian group $G$, it seems reasonable that the function $\sin(ax)$ can provide an example on the disk for the original Pompieu  problem. On the other hand, 
\out{it is surprising that exponential solutions were not found in
literature.}
\end{rem}

\subsection{Connection of the problems}

\begin{prop}
Let $G$ be a finite abelian group. If $S \subset G$ is a spectral
set with $|S|\ge 2$, then $S$ is a non-Pompeiu set.

\end{prop}

\begin{proof}
The spectral property of $S$ requires a set of irreducible
representations, of the same cardinality of $S$, whose
\out{restrictions} to $S$ are pairwise orthogonal. Assume $\chi$ and
$\psi$ are different irreducible representations of $G$, whose
restriction to $S$ are orthogonal. Since $[\chi_{|S},
\psi_{|S}]=[(\chi \bar{\psi})_{|S},id_{|S} ]$ we obtain a
representation $\phi=\chi \bar{\psi}$ such that $\sum_{s \in S}
\phi(s)=0$, which leads us back to the Pompeiu problem. Thus we get
that $S$ is a non-Pompeiu set with respect to the irreducible
representation $\phi$.
\end{proof}

\section{Pompeiu problem for cyclic groups}\label{sectionPompfor abelian}
In this section we consider the non-Pompeiu sets for abelian groups.

Every representation of a finite abelian group is linear, so it
factorizes through a faithful representation of a cyclic group since
the finite subgroups of $\mathbb{C}\setminus \{0\}$ are cyclic. This
shows that some sort of description for non-Pompeiu sets of finite
abelian groups is given by understanding the non-Pompeiu weight
functions of cyclic groups with respect to faithful representations.

Let $(\Z_N,+)$ be the cyclic group of order $N$.
Note that for all $k\mid N$ there is a unique normal subgroup
$\Z_k\leq \Z_N$ of order $k$.
The group generated this way contains exactly the elements of $\Z_N$
divisible by $\frac{N}{k}$ so this subgroup of $(\Z_N,+)$ will also
be denoted by
$H_{\frac{N}{k}}$. 

We use the following isomorphism between $\mathbb{Z}_N$ and
$\widetilde{\Z}_N$: fix a primitive $N$'th root of unity $\alpha$, a
generator $g$ of $\mathbb{Z}_N$. Then for any $j \in \mathbb{Z}_N$
the function $\psi_j(g^{i})=\alpha^{ji}$ gives a homomorphism from
$\mathbb{Z}_N$ to $\mathbb{C}^*$ hence it is an irreducible
representation. Now $j \to \psi_j$ gives the isomorphism from
$\mathbb{Z}_N$ to $\widetilde{\Z}_N$; \out{throughout the text, we
will use the isomorphism that arises from $\alpha=\xi_N$}. From now
on the subgroup of $\widetilde{\Z}_N$ isomorphic to $H \le
\mathbb{Z}_N $
will be denoted by $\widetilde{H}$. 

Hereinafter we use the notion of mask polynomial.
\begin{de}\label{defmask}
Let $G$ be a cyclic group and $w:G\to \Q$ be a weight function. We
call $$m_{w}(x)=\sum_{h \in G} w(h)x^h$$ the \textit{mask
polynomial} of $w$, where $w(h)$ denotes the weight of $h\in G$.
This might be considered as an element of $\Q[x]/(x^n-1)$. For a
(multi)set $S$ of $G$ we define the \textit{mask polynomial} of $S$
by $$S(x)=\sum_{s \in S} c_sx^s,$$ where $c_s$ denotes the
cardinality of $s\in S$.

\end{de}

Let $\Phi_{k}(x)$ denote the $k$'th cyclotomic polynomial, which is of degree $\varphi(k)$, where $\varphi$ denotes the Euler totient function. Note that for \out{fixed} $N$ and prime $p\mid N$ the mask polynomial of 
$\Z_p\leq \Z_N$ is $\Php$. The following is one of the key
preliminary observations. Basically, this can be considered as a
statement on roots of unity. There is a \out{vast} literature on
vanishing sums of roots of unity. This particular statement gives a
generalization of Theorem 3.3 of \cite{vanishingsum}. Similar
results might appear in many other papers.

\begin{prop}\label{propsqfree}
Let $G$ be a cyclic group of order $N$ and let $\alpha$ be a
primitive $N$'th root of unity.  We denote by $P_N$ the set of prime
divisors of $N$. Further let $w$ be a weighted function. Then $w$ is
non-Pompeiu with respect to the faithful representation
$\psi_{\alpha}$ if and only if $$w=\sum_{g \in G} \sum_{p \in P_N}
w_{p,g} 1_{\mathbb{Z}_p+g}$$ for some $w_{p,g} \in \mathbb{Q}$,
where $1_{\mathbb{Z}_p+g}$ denotes the characteristic function of
the coset $\mathbb{Z}_p+g$.
\end{prop}
\begin{proof}
The fact that $w$ is a non-Pompeiu weight function with respect to
the faithful representation $\psi_{\alpha}$ means that $\alpha$ is
the root of the mask polynomial $m_w$ of $w$, since
$m_w(\alpha)=\sum_{i=0}^{N-1} w(i)\alpha^{i}=0$. On the other hand,
for a given $N \in \mathbb{N}$ every $p \in P_N$ we have that
$\alpha$ is the root of the mask polynomial of
$\mathbb{Z}_p \le \mathbb{Z}_N$, that is $\Php$. Indeed, $\alpha$ is
a primitive $N$'th root of unity so $\alpha^{\frac{N}{p}} \ne 1$.
Clearly,
$\alpha^{\frac{N}{p}}\Phi_p(\alpha^{N/p})=\Phi_p(\alpha^{N/p})$, so
it implies $\Phi_p(\alpha^{N/p})=0$.

Then $\alpha$ is also a root of the polynomial $m_w(x)+\sum_{p \in
P_N} a_p(x) \Php$, where $a_p(x) \in \mathbb{Q}[x]$. By using
Euclidean division there are polynomial $q(x),r(x) \in
\mathbb{Q}[x]$ such that $$m_w(x)=q(x) \Phi_N(x)+r(x)$$ with either
$r(x)$ to be the constant zero function or $deg(r(x)) < \varphi(N)$.

The common roots of the polynomials $\Php \ (p\in P_N)$ are exactly
the primitive $N$'th roots of unity. The multiplicity of these roots
in all of these polynomials is $1$. These polynomials are all in
$\Q[x]$ so the greatest common divisor in the ring $\mathbb{Q}[x]$
of the polynomials $\Php \ (p\in P_N)$ is $\Phi_N(x)$. Thus
$$\Phi_N(x)=\sum_{p \in P_N} a_p(x) \Php$$ (with some $a_p(x) \in
\mathbb{Q}[x]$). Substituting this to the previous equation we
obtain that $m_w(x)-\sum_{p \in P_N} q(x) a_p(x) \Php$ is of degree
less than $\varphi(N)$ or $m_w(x)-\sum_{p \in P_N} q(x) a_p(x) \Php$
is the constant zero function. Since $\Phi_N(x)$ is the minimal
polynomial of $\alpha$ over $\mathbb{Q}$, we have $m_w(x)-\sum_{p
\in P_N} q(x) a_p(x) \Php=0$. Thus $$m_w(x)=\sum_{p \in P_N} q(x)
a_p(x) \Php.$$ It is clear that $x^k\Php$ is the mask polynomial of
a coset of $\mathbb{Z}_p$ for every $0\le k<N$. Hence we have
$$w(x)=\sum_{g \in G} \sum_{p \in P_N} w_{p,g}
1_{\mathbb{Z}_p+g}(x)$$ with some $w_{p,g} \in \mathbb{Q}$.

The other direction follows from the fact that
$\Phi_p(\alpha^{N/p})=0$ for every $p \in P_N$.
\end{proof}
We note that using Proposition \ref{propsqfree} one can simply
construct the asymmetric minimal sums of roots of unity appearing
in \cite{vanishingsum}.

In terms of mask polynomials the previous proposition can be stated
as follows.
\begin{cor}\label{corpqfree}
Let $S(x)\in\Z_{\geq0}[x]$ with $S(\xi_N)=0$, where
$N=p_1^{m_1}\cdots p_n^{m_n}$ and $p_1,\dots ,p_n$ are primes. Then,
 \[S(x)\equiv P_1(x)\Phi_{p_1}(x^{N/p_1})+\ldots+ P_n(x)\Phi_{p_n}(x^{N/p_n})\bmod(x^N-1),\]
 for some $P_1(x),\dots, P_n(x)\in\Q[x]$. \end{cor}

The following is an easy consequence of Proposition
\ref{propsqfree}.

\begin{cor}\label{corradn}
Let $G$ be a cyclic group of order $N$ and 
$\Psi$ be a faithful representation of $G$. Assume $w$ is a
non-Pompeiu weight function with respect to $\Psi$. Then the
restriction of $w$ to each $\Z_{\rad(N)}$-coset is the weighted sum
of characteristic functions of $\Z_{p_i}$-cosets, where $\rad(N)$
denotes the square free radical of $N$.
\end{cor}

We will consider $\mathbb{Z}_{\prod_{i=1}^d p_i}\cong
\prod_{i=1}^d\Z_{p_i}$ as a grid in $\mathbb{R}^d$, whose points
have integer coordinates. More precisely for
$\mathbb{Z}_{\prod_{i=1}^{d} p_i}$ we assign
\[ \mathcal{G}=\{ x \in \Z^d \mid 0 \le x_i \le p_i-1 ~~ \mbox{for } 1 \le i \le d\},
\]
where $x_i$ denotes the $i$'th coordinate of $x$. The cosets of
$\mathbb{Z}_{p_i}$ \out{coincide with collections} of parallel line
segments (containing $p_i$ grid points of $\mathcal{G}$. A $d$
dimensional grid-cuboid will be a collection of $2^d$ grid
points, whose convex hull forms a $d$-dimensional cuboid 
in $\R^d$. Let $P \subset \mathcal{G}$ be a $d$-dimensional
grid-cuboid 
and fix a point $y \in P$. For a point of $z \in P$ let $\pi(z)$
denote the Hamming distance between $z$ and $y$. Note that $w$ can
also be considered as a function from $\mathcal{G}$ to $\Q$.

The following statement makes the Pompeiu property for weight
functions easily recognizable.
\begin{prop}\label{prophypercube}
Let $w$ be a non-Pompeiu weight function on the set of
$\mathbb{Z}_{\prod_{i=1}^d p_i}$, where $p_i$ are mutually different
primes. If $w$ is the weighted sum of characteristic functions of
$\Z_{p_i}$-cosets, then 
for every $d$ dimensional
grid-cuboid 
$P$ we have
\begin{equation}\label{eqparallelograme}
\sum_{c \in P} (-1)^{\pi(c)}w(c)=0.
\end{equation}
\end{prop}
\begin{proof}
It is easy to see that each coset of $\Z_{p_i}$ for any $p_i\mid n$
contains either 2 or 0 elements of the cuboid $P$. Substituting the
characteristic function of any coset of $\Z_{p_i}$ as a weight
function to the left-hand side of \eqref{eqparallelograme}, it is
clearly reduced to a sum of at most two elements with different
sign, thus \eqref{eqparallelograme} holds.
\end{proof}
\begin{rem}
The converse of the previous statement also holds.  We leave it to the reader to work out the details of the proof . 
\end{rem}

Now we describe a few special cases which will be later used for the
proof of Theorem \ref{thmfotetel}. In the proof of the next
proposition we use the following definition.
\begin{de}
Let $S\ssq \Z_N$. For every $j\in\Z$ and $d\mid N$, we define the
following subsets
 \[S_{j\bmod d}=\set{s\in S:s\equiv j\bmod d}.\]
\end{de}


\begin{prop}\label{proppq}
\hspace{2em}
\begin{enumerate}
\item\label{itempqa} Every non-Pompeiu set in $\mathbb{Z}_{pq}$ with respect to a faithful representation is either the union of
cosets of $\mathbb{Z}_p$ or those of $\mathbb{Z}_q$.
\item\label{itempqb} Let $N=p^mq^n$ and let $S$ be a non-Pompeiu multiset in $\Z_{N}$ with respect to a faithful representation. Then there are some polynomials $P(x), Q(x)\in \Z_{\ge 0}[x]$ such that
 \[S(x)\equiv P(x)\Php+Q(x)\Phq\bmod(x^N-1).\]
\end{enumerate}
\end{prop}
\begin{proof}
\begin{enumerate}
\item
Let $S$ be a non-Pompeiu set in $\Z_{pq}$ with respect to a faithful
representation and let $w$ be the characteristic function of $S$.
Using Proposition \ref{propsqfree} we might write
$w=\sum_{i=0}^{q-1}a_i 1_{\Z_p+i}+\sum_{j=0}^{p-1}b_j 1_{\Z_q+j}$,
where $a_i, b_j \in \mathbb{Q}$. Then the range of $w$ is $Ran(w)=
\{ a_i+b_j \mid ~ 0 \le i \le p-1,~ 0 \le j \le q-1 \}$. We have
that $Ran(w)= \{0,1 \}$. Thus there are at most two different $a_i$
and two different $b_j$.

One can treat the case when $a_i$ and $b_j$ are constants as a function of $i$ and $j$, respectively. 
Thus we may assume that $a_k < a_l$ for some $0 \le k,l \le p-1$.
Then clearly $a_k+b_j=0$ and $a_l+b_j=1$ for all $b_j$, in
particular all $b_j$'s are the same. Therefore, we may write
\[ w= b+ \sum_{i=0}^{p-1}a_i 1_{\Z_p+i}= \sum_{i=0}^{p-1}(b+a_i) 1_{\Z_p+i},\]
finishing the proof of the statement.

\item By Corollary \ref{corpqfree}, it is clear that \[S(x)\equiv P(x)\Php+Q(x)\Phq\bmod(x^N-1)\]  for some $P(x), Q(x)\in \Q[x]$. Now we show that $P$ and $Q$ can be chosen such that $P(X), Q(x)\in \Z_{\ge 0}[x]$.

The subgroups $\Z_p$ and $\Z_q$ generate $\mathbb{Z}_{pq}$. Thus $S$
can be written as the disjoint union $$S= \cup_{k \in C} S_{k \bmod
N/pq}$$ for $k=0\stb N/pq-1$, where $k$ runs through a set of
representatives $C$ of the cosets of $\Z_{pq}$. Thus we are given
the following: $$S_{k \bmod N/pq}=\sum_{a \in A} c_a
(\Z_p+a)+\sum_{b \in B} d_b (\Z_q+b),$$ where $c_a+ d_b \in \Z_{\ge
0}$ and $A$ and $B$ are sets of coset representatives of $\Z_p$ and
$\Z_q$, respectively, in $\Z_{pq}+k$. We want to modify the
coefficients $c_a, d_b$ so that they produce the same multiset but
all of them are nonnegative.

Let $e=c_a+d_b$ be one of the minimal weights of the multiset $S$.
Then the values $d_x'=(c_a+d_x)-e$ are nonnegative for every $x \in
B$ and let $c_y'=c_y+d_b$, which are nonnegative since these values
are given by the multiset $S$ only.

Now $c_y'+d_x'=( (c_a+d_x) -e )+c_y+d_b=c_y+d_x$ for every $x \in B$
and $y \in A$, finishing the proof of the lemma. \out{\qedhere}


\end{enumerate}
\end{proof}

\section{Reduction (of Fuglede's conjecture)}\label{section5}

\






Before we proceed to the proof of Theorem \ref{thmfotetel} we make a
few general observations.
\begin{lem}\label{lemeltol}
Let $G$ be a finite abelian group. Assume that $S \subset G$ is a
spectral set having $\Lambda$ as a spectrum.
\begin{enumerate}
\item\label{itemeltols} $S+t$ is spectral with the same spectrum $\Lambda$ for every $t \in G$.
\item\label{itemeltolomega} $\Lambda+ \omega$ is a spectrum for $S$ for every $\omega \in G$.
\item\label{itemduality} $S$ is a spectrum for $\La$.
\end{enumerate}
\end{lem}
\begin{proof}
\begin{enumerate}
\item If $\sum_{s \in S} \chi_{\delta}(s)=0$ for some $\delta \in \lala$, then since $\chi_{\delta}$ is a homomorphism, we have $$\sum_{u \in (S+t)} \chi_{\delta}(u)= \chi_{\delta}(t)\sum_{s \in S} \chi_{\delta}(s)=0.$$ Thus the orthogonality of the representations corresponding to the spectrum is preserved under translation.

\item Similarly, the orthogonality of the representations corresponding to $\La + \omega$ follows from the fact that
$\lala=(\La + \omega)-(\La + \omega) $.
\item This follows by the fact that a finite abelian group is canonically isomorphic to its double dual. \qedhere
\end{enumerate}
\end{proof}

\begin{cor}
It is enough to prove Theorem \ref{thmfotetel} for spectral sets $S$
with $0 \in S$ and with spectrum $\Lambda$ that contains $0$.
\end{cor}
From now on we assume $0 \in S$ and $0 \in \La$.

\begin{lem}\label{propgenerate}
Let $G$ be a finite abelian group and let $S$ be spectral in $G$,
that does not generate $G$. Assume that for every proper subgroup
$H$ of $G$ we have $\mathbf{S-T}$(H). Then $S$ tiles $G$.
\end{lem}
\begin{proof}
Let $S$ be a spectral set with orthogonal basis $\{\chi_{\lambda} :
\lambda \in \Lambda \}=\chi_{\Lambda} \subset \widetilde{G}$ and let
$\langle S \rangle =H < G$. Since every $\chi_{\lambda}$ is $1$
dimensional, we have $\{ {\chi_{\lambda}}_{|H} \mid \lambda \in
\Lambda \} \subseteq \widetilde{H}$ and clearly these are still
orthogonal on $S$, since $S \subset H$. Then using that
$\mathbf{S-T}$(H) holds, there is a set $T \subset H$ with $S+T=H$.
Now let $U$ be a complete set of coset representatives of $G/H$.
Then we have $S+(T+U)=G$.
\end{proof}
Now we prove a similar lemma reducing the possible structure of
$\La$.
\begin{lem}\label{lem1} Let $G$ by a cyclic group of order $N$ and let us suppose that $\mathbf{S-T}$(G/H) holds on every proper factor $G/H$.
Let $S$ be a spectral set of $G$ and $\Lambda$ be the corresponding
spectrum. Assume that the intersection of the kernels of the
elements of $\chi_{\La}$ contains $H_{\frac{N}{\ell}} \ne 1$ for
some $1<\ell \mid N$. Then $S$ tiles $G$.
\end{lem}
\begin{proof}
By our assumptions that the elements of $\chi_{\Lambda}$
can be considered as irreducible representations of $G/H_{\frac{N}{\ell}}$ since their kernel \out{is contained} in $H_{\frac{N}{\ell}}$. 

Let $S_{\ell}$ denote the multiset obtained as the image of $S$ by
the canonical projection $\pi_{\ell}$ of $G$ to
$G/H_{\frac{N}{\ell}}  \cong
H_{\ell}$. We claim that multiset $S_{\ell}$ is a set in $H_{\ell}$. 
Indeed there can not be two elements of $S$ in the same coset of
$H_{\frac{N}{\ell}}$ since otherwise each element of
$\chi_{\Lambda}$ would have the same value on them, contradicting
the fact that these representations form a basis of the set of
complex valued function on $S$. Thus $S_{\ell}$ is a
set. 
Now it is easy to derive that $\Lambda/H_{\frac{N}{\ell}}$ is a
spectrum with respect to $S_{\ell}$ in $G/H_{\frac{N}{\ell}}$ since
$\chi_{\lambda}(\pi_{\ell}(s))=\chi_{\lambda}(s)$ for every $s \in
S$ and $\la \in \La$.

We know $\mathbf{S-T}$($G/H_{\frac{N}{\ell}}$) holds. As $S_{\ell}$
is a spectral set in $G/H_{\frac{N}{\ell}}$ there is $T_{\ell}
\subset G/H_{\frac{N}{\ell}}$ with
$S_{\ell}+T_{\ell}=G/H_{\frac{N}{\ell}}$. Then if $T$ is the
preimage of $T_l$ under the canonical projection from $G$ to
$G/H_{\frac{N}{\ell}}$, then we have $S+T=G$.
\end{proof}


\begin{obs}\label{obsorder} Let us recall that $S(x)$ is the mask polynomial of the spectral set $S$.
Note that for $\chi \in \widetilde{G}$ of order $k$, then $\sum_{s
\in S}\chi(s)=0$ is equivalent to the fact that a primitive $k$'th
root of unity $\xi_k$ is a root of $S(x)$. Since $\Phi_k(x)$ is
irreducible over $\Q$ we have $\Phi_k(x) \mid S(x)$ hence every
primitive $k$'th root of unity is the root of $S(x)$ and $\sum_{s
\in S}\chi'(s)=0$ for every $\chi' \in \widetilde{G}$ of the same
order. If $\La\subseteq G$ is a spectrum of $S$, the above can be
summarized to
\begin{equation}\label{ordroots}
 S(\xi_{\ord(\la-\la')})=0,
\end{equation}
for every $\la\neq\la'$ in a spectrum $\La$, using \eqref{Ladiff}.
\end{obs}

The question whether our techniques can be generalized naturally
arises. We point out here that in the next proposition we heavily
use the assumption that the order of cyclic groups is divisible by
at most two different primes.
\begin{prop}\label{propfaithful}
Let $G$ be a cyclic group of order $p^kq^{\ell}$ and let $|S| \ge 2$
be a spectral set.
Assume further that $\La$ is a
spectrum for $S$ such that the elements of $ \chi_{\Lambda} 
$ do not have a nontrivial common kernel.
Then for every faithful representation $\psi$ of $G$ we have
$\sum_{s \in S}\psi(s)=0$.
\end{prop}
\begin{proof}
Note that by Observation \ref{obsorder}, it is enough to prove the
statement for one faithful representation.

Since the elements of $\chi_{\La}$ do not have a common kernel we
have a $\la_1 \in \La$ with $p \nmid \la_1$. If $q \nmid \la_1$,
then we are done so we assume $q \mid \la_1$. Similarly, we might
assume that there exists $\la_2 \in \La$ with $q \nmid \la_2$ and $p
\mid \la_2$. In this case $\chi_{\lambda_1-\lambda_2}$ generates
$\widetilde{G}$ so we have $\sum_{s \in S}\chi_{\la_1-\la_2}(s)=0$.
\end{proof}

This has the following interpretation in terms of mask polynomials.

\begin{cor}\label{Nroot}
 Let $(S,\La)$ be a spectral pair in $\Z_N$, where $N=p^kq^{\ell}$, such that $0\in S$, $0\in \La$, and each of $S$, $\La$ generates $\Z_N$. Then
 \[S(\xi_N)=\La(\xi_N)=0.\]
\end{cor}

\begin{prop}\label{propcsakzpeltolt}
Let $S$ be a spectral set in $\Z_N$ and let $p$ be a prime divisor
of $N$. Assume that for every proper factor group $\Z_N/H$ of $\Z_N$
we have $\mathbf{S-T}(\Z_N/H)$. Assume further that $S$ is the
disjoint union of cosets of $\mathbb{Z}_p$. Then $S$ tiles
$\mathbb{Z}_{N}$. \end{prop}

\begin{proof}
By our assumptions $\abs{S}=pr=\abs{\Lambda}$ for some $r \in
\mathbb{N}$ and $\La$ is a spectrum for $S$. Thus at least one of
the cosets of $H  _{p}$ contains at least $r$ elements of $\Lambda$.
By Lemma \ref{lemeltol} \ref{itemeltolomega} we may assume that
$\abs{H_{p} \cap  \La  } \ge
r$. The elements $\chi_{\La}
\subseteq \widetilde{H}_{p}$ are representations
    having a common kernel
$\mathbb{Z}_p=H_{\frac{N}{p}}$. By our assumption $S$ is the
disjoint union of $\Z_p$-cosets, so it can be written as
$\mathbb{Z}_p + B$ for some $B \subseteq \mathbb{Z}_N/\Z_{p}$. The
representations in $\widetilde{H}_{p} \cap \chi_{\La}$ are constant
on every coset of $\mathbb{Z}_p$. Hence for every $\chi_1 \ne \chi_2
\in \widetilde{H}_{p} \cap \chi_{\La}$ we have
\begin{equation*}
\begin{split}
0&= \sum_{s \in S} \chi_1(s)\bar{\chi}_2(s)= \sum_{s \in \Z_p +B} \chi_1(s)\bar{\chi}_2(s)=\sum_{t \in B} \sum_{x \in \Z_p}\chi_1(t+x)\bar{\chi}_2(t+x) \\
&=\sum_{t \in B} \sum_{x \in
\Z_p}\chi_1(t)\chi_1(x)\bar{\chi}_2(t)\bar{\chi}_2(x)=\sum_{t \in
B}p\chi_1(t)\bar{\chi}_2(t)=p \sum_{t \in
B}\chi_1(t)\bar{\chi}_2(t),
\end{split}
\end{equation*}
since the kernel of $\chi_1$ and $\chi_2$ contains $\Z_p$. Thus we
obtain a set of $r=|B|$ representations of $\Z_N/\Z_p$, which are
mutually orthogonal, hence forming a basis of $L^2(B)$.
Thus $B$ is a spectral set in $\Z_N/\Z_p$ and using our
assumption %
we obtain that there exists $T$ with $B+T=\Z_N/\Z_p$. So finally we
get $S+T= (\Z_p+B)+T=\Z_p+(B+T)=\Z_p+ \Z_N/\Z_p=\mathbb{Z}_N$.
\end{proof}

Before we start to detail the proof of Theorem \ref{thmfotetel} we
summarize that we have already proved in the previous sections about
the structure of a spectral set $S$ in $\mathbb{Z}_{p^nq^2}$. Note
that we may assume by induction on $n$ that $\mathbf{S-T}(H)$ holds
for every proper subgroup or factor $H$ of $\Z_{p^nq^2}$. Indeed,
Fuglede's Conjecture holds for $\Z_{pq^2}$ and for $\Z_{p^nq}$ by
\cite{MK17}, which corresponds to the base case of our induction.

If $|S|=1$, then $S$ is clearly a spectral set and also a tile. By
Lemma \ref{lem1} we might assume that the elements of $\chi_{\La}$
do not have a common kernel so by Proposition \ref{propfaithful} we
might assume that $|S| \ge 2$ is a non-Pompeiu set with respect to a
faithful representation of $\mathbb{Z}_{p^nq^2}$. Hence
by Proposition \ref{propsqfree} we have $$S=\sum_{g \in A} u_g
(\Z_p+g)+\sum_{h \in B} v_h (\Z_q+h),$$ where $u_g, v_h \in \Q$ and
$A$ and $B$ are sets of coset representatives of $\Z_p$ and $\Z_q$,
respectively. Thus $S$ is the weighted sum of cosets of
$\mathbb{Z}_p$ and $\mathbb{Z}_q$. Until now we have only seen that
the weights are rational numbers. Now we prove that all weights are
$0$ or $1$.

The subgroups $\Z_p$ and $\Z_q$ generate $\mathbb{Z}_{pq}$, so we
write $S$ as the disjoint union $$S= \cup_{k \in C} S_{k \bmod
N/pq},$$ where $k$ runs through a set of representatives $C$ of the
cosets of $\Z_{pq}$ for $k=0\stb N/pq-1$. Now
\begin{equation}\label{eqSk}
S_{k \bmod N/pq}= \sum_{g \in A,~g+\Z_p \subset k+\Z_{pq}    } u_g
(\Z_p+g)+\sum_{h \in B,~h+\Z_q \subset k+\Z_{pq}  } v_h (\Z_q+h)
\end{equation}
for every $k \in C$, so $S_{k \bmod N/pq}$ inherits its weights from
$S$. Now it follows from Proposition \ref{proppq} that in
\eqref{eqSk} $u_g=0$ for every $g \in A,~g+\Z_p \subset k+\Z_{pq}$
or $v_h=0$ for every $h \in B,~h+\Z_q \subset k+\Z_{pq}$. Since
$S_{k \bmod N/pq}$ is a set, the remaining coefficients are $0$ or
$1$. Then $S_{k \bmod N/pq}$ is the disjoint nontrivial union of
$\Z_p$-cosets or $\Z_q$-cosets. Only one type appears for every
fixed $k=0 \stb N/pq-1$ except in the obvious case as follows:

It can happen that $S$ contains a whole $\Z_{pq}$-coset, in which
case it can be considered as the union of only $\Z_p$-cosets and
only $\Z_q$-cosets as well.
Thus $S$ is the disjoint union of $\Z_p$-cosets and $\Z_q$-cosets.

Beside the case when $S$ contains both $\mathbb{Z}_{p}$-cosets and
$\mathbb{Z}_{q}$-cosets, by Proposition \ref{propcsakzpeltolt} we
are done. Thus, we may assume $S$ contains both
$\mathbb{Z}_{p}$-cosets and $\mathbb{Z}_{q}$-cosets; we shall call
such sets \emph{nontrivial unions} of $\Z_p$- and $\Z_q$-cosets, to
emphasize that they cannot be expressed as unions consisting solely
of $\Z_p$-cosets, or $\Z_q$-cosets.

The above also follows from Corollary \ref{Nroot} and the structure
of vanishing sums of roots of unity of order $N$, where $N$ has at
most two distinct prime factors \cite{vanishingsum}. We added also a
condition that shows when such a vanishing sum corresponds to a
nontrivial union of $\Z_p$- and $\Z_q$-cosets, which is a
consequence of Corollary \ref{corpqfree} and Proposition
\ref{proppq}\ref{itempqb} or alternatively of Proposition 2.6 in
\cite{MK17}.

\begin{Thm}\label{vansumspos}
 Let $F(x)\in\Z_{\geq0}[x]$ and $N=p^mq^n$, where $p, q$ are different primes. Then, $F(\xi_N)=0$ if and only if
 \[F(x)\equiv P(x)\Php+Q(x)\Phq\bmod(x^N-1),\]
 for some $P(x),Q(x)\in\Z_{\geq0}[x]$. If $F(\xi_N^{p^k})\neq0$ (respectively, $F(\xi_N^{q^{\ell}})\neq0$) for some $1\leq k\leq m$ (resp. $1\leq \ell\leq n$), then we cannot have $P(x)\equiv0\bmod(x^N-1)$
 (resp. $Q(x)\equiv0\bmod(x^N-1)$).
\end{Thm}

We will repeatedly use the above in Section \ref{sectiont2} in order
to obtain information about the structure of $S$ and $\La$ from the
vanishing of their mask polynomials on various $N$'th roots of
unity. Regarding the case when $S$ is a union of $\Z_p$-cosets (or
$\Z_q$-cosets), there is a characterization in terms of the mask
polynomial. This follows from a special case of Ma's Lemma \cite{Ma}
(see also Lemma 1.5.1 \cite{SchmidtFDM}, or Corollary 1.2.14
\cite{Pott}), adapted to the cyclic case, using the polynomial
notation.

\begin{lem}\label{ma}
 Suppose that $S(x)\in\Z[x]$, and let $\Z_N$ be a cyclic group such that $p^m\mid N$, but $p^{m+1}\nmid N$. If $S(\xi_d)=0$, for every $p^m\mid d\mid N$, then
 \[S(x)\equiv P(x)\Php\bmod(x^N-1).\]
 If the coefficients of $S$ are nonnegative, then $P$ can be taken with nonnegative coefficients as well. In particular, if $S\subseteq\Z_N$ satisfies $S(\xi_d)=0$, for every $p^m\mid d\mid N$, then
 $S$ is a union of $\Z_p$-cosets.
\end{lem}
We summarize the reductions made so far in the following list.


\begin{summary}\label{summary}
We might assume that a spectral set $S \subset \mathbb{Z}_{p^nq^2}$
along with a spectrum $\La$, have the following structure:
\begin{enumerate}\label{enumsummary}
\item\label{enumsummary0} $0\in S$, $0\in \La$ and each of $S$ and $\La$ generates $\Z_{p^nq^2}$.
\item\label{enumsummarya} Both $S$ and $\La$ can be written as the disjoint nontrivial union of $\Z_p$-cosets and $\Z_q$-cosets and this holds for $S\cap (\Z_{pq}+g)$ and $\La\cap (\Z_{pq}+h)$ for every $g,h \in \Z_{p^nq^2}$ as well.
\item\label{enumsummaryb} There is a $\Z_{pq}$-coset which intersects $S$ and its complement. Further the intersection is the union of $\Z_p$-cosets. The same holds for another $\Z_{pq}$-coset with
$\Z_q$-cosets as well.
\item\label{enumsummaryc} Fuglede's conjecture holds for all $\Z_M$, with $M\mid p^nq^2$, $M<p^nq^2$ (induction assumption).
\end{enumerate}
\end{summary}
\begin{proof}
\begin{enumerate}
\item Follows from Lemma \ref{lemeltol} and Proposition \ref{propgenerate}.
\item Immediate consequence of part (a), Lemma \ref{propsqfree}, Proposition \ref{proppq} and Corollary \ref{Nroot}.
\item Follows from Proposition \ref{propcsakzpeltolt}.
\item It was proved in \cite{MK17} that Fuglede's conjecture holds for $N=p^nq$, and also for $N=pq^2$, so the given
 statement certainly holds for $p^2q^2$, which is the base case for the inductive argument. \qedhere
\end{enumerate}
\end{proof}

Now we turn to the main tool already used in \cite{MK17} to prove
that a spectral set tiles $\mathbb{Z}_{p^nq^2}$.
Clearly, sets coincide with mask polynomials having only
coefficients $0$ and $1$.
The following theorem was proved in \cite{CM1999}. Let $H_S$ be the
set of prime powers $r^a$ dividing $N$ such that $\Phi_{r^a}(x) \mid
S(x)$.
\begin{Thm}\label{thmt1t2}
If $S \subset \mathbb{Z}_N$ satisfies the following two conditions
\ref{t1} and \ref{t2}, then $S$ tiles $\mathbb{Z}_N$.
\begin{enumerate}[{\bf(T1)}]
 \item\label{t1} $S(1)=\prod_{d \in H_S} \Phi_d(1)$.
 \item\label{t2} For pairwise relative prime elements $s_i$ of $H_S$, $\Phi_{\prod s_i} \mid S(x)$.
\end{enumerate}
\end{Thm}
Note that $\Phi_{p^a}(1)=p$ for a prime $p$ and $\Phi_{k}(1)=1$ if
$k$ has at least two different prime divisors.

\section{Preliminary lemmas}\label{section6}
We introduce an extra notation for divisibility. Fix $N \in
\mathbb{N}$. For a natural number $k$ we write $\ell\pdiv_N k$ if
$\ell$ is the largest divisor of $N$, which divides $k$. In our case
$N$ will be $p^nq^2$ so we simply write $\ell\pdiv k$.

We review first the equations \eqref{Ladiff} and \eqref{ordroots}
for a spectral pair $(S,\La)$ in $\Z_N$. First, we define as usual
\[\Z_N^{\star}=\set{g\in \Z_N:\gcd(g,N)=1},\]
the group of \emph{reduced residues} $\bmod N$. It is precisely the
subset of elements of $N$ of order exactly $N$. Similarly, the
subset of $\Z_N$ of elements of order $N/d$, where $d\mid N$, is
\[d\Z_N^{\star}=\set{g\in\Z_N:\gcd(g,N)=d}.\]
The zero set \[Z(S)=\set{d\in\Z_N: S(\xi_N^d)=0}\] is then a union
of subsets of the form $d\Z_N^{\star}$, for some $d\mid N$, and
\eqref{Ladiff} and \eqref{ordroots} can be rewritten as
\begin{equation}\label{divclass}
\La-\La\ssq\set{0}\cup\bigcup_{d\mid N, S(\xi_N^d)=0}d\Z_N^{\star}.
\end{equation}
Of course, by Lemma \ref{lemeltol}\ref{itemduality}, the roles of
$S$ and $\La$ can be reversed.

\begin{de}
 Let $S\ssq \Z_N$. Recall that for every $j\in\Z$ and $d\mid N$, we define the following subsets
 \[S_{j\bmod d}=\set{s\in S:s\equiv j\bmod d}.\]
 We say that $S$ is \emph{equidistributed} $\bmod d$, if
 \[\abs{S_{j\bmod d}}=\frac{1}{d}\abs{S},\]
 for every $j$. Equivalently, every $\Z_{N/d}$-coset of $\Z_N$ contains the same amount of elements of $S$.
\end{de}

\begin{lem}\label{lemegyenessav}
\hspace{2em}
\begin{enumerate}
\item\label{itemegyenessava} Assume $\Phi_p(x)\mid S(x)$. Then every
$\Z_{N/p}$-coset of $\Z_N$ contains the same amount of elements of
$S$.
\item\label{itemegyenessavb} Assume $\Phi_k(x)\mid S(x)$ for every $1<k\mid d$. Then every
$\Z_{N/d}$-coset of $\Z_N$ contains the same amount of elements of
$S$.
\end{enumerate}
\end{lem}
\begin{proof}
\begin{enumerate}
 \item $\Phi_p(x) \mid S(x)$ is equivalent to the fact that $S$ is a non-Pompeiu set with respect to an irreducible representation of order $p$, whose kernel is $\mathbb{Z}_{N/p}$.
 It is easy to see that a non-Pompeiu multiset on $\mathbb{Z}_p$ has to be constant we obtain the result.
\item Consider the formula
 \begin{equation}\label{Smodd}
  S(x)\equiv \sum_{j=0}^{d-1}\abs{S_{j\bmod d}}x^j\bmod(x^d-1),
 \end{equation}
 which holds for every $S\ssq\Z_N$. It holds $S(\xi_k)=0$ for every $1<k\mid d$ if and only if
\[1+x+\dotsb+x^{d-1}=\prod_{1<k\mid d}\Phi_k(x)\mid S(x),\]
or equivalently $S(x)=(1+x+\dotsb+x^{d-1})G(x)$. The latter implies
\[S(x)\equiv (1+x+\dotsb+x^{d-1})G(1)\bmod (x^d-1),\]
so by \eqref{Smodd}, we get $\abs{S_{j\bmod d}}=G(1)$ for all $j$.
Conversely, if $\abs{S_{j\bmod d}}=c$ for all $j$, then
\[S(x)\equiv c(1+x+\dotsb+x^{d-1})\bmod (x^d-1),\]
due to \eqref{Smodd}, which easily gives $S(\xi_k)=0$ for every
$1<k\mid d$, as desired. \qedhere
\end{enumerate}
\end{proof}

Let $(S,\La)$ be a spectral pair in $\Z_N$ satisfying the conditions
of Reduction \ref{summary}, where $N=p^nq^2$. An immediate
consequence of Reduction \ref{summary}\ref{enumsummaryb} is that
$S-S$ contains the difference set of both a $\Z_p$-coset and a
$\Z_q$-coset, thus
 \[\frac{N}{p}\Z_N\cup\frac{N}{q}\Z_N\subseteq S-S,\]
 whence
 \begin{equation}\label{Bpqroots}
\La(\xi_p)=\La(\xi_q)=0,
\end{equation}
by \eqref{ordroots},
 and we obtain in particular, 
\begin{equation}\label{ed}
\abs{\La_{i\bmod p}}=\frac{1}{p}\abs{\La},\ \abs{\La_{j\bmod
q}}=\frac{1}{q}\abs{\La},
\end{equation}
for all $i,j$, by Lemma \ref{lemegyenessav}. This shows that $pq$
divides $\abs{S}=\abs{\La}$.

\section{Proof of Theorem \ref{thmfotetel}}\label{sectiont2}

A significant special case will be shown first.

\begin{lem}\label{q2divides}
Let $S\subseteq\Z_N$ be spectral. If $q^2\mid \abs{S}$, then $S$
tiles $\Z_N$.
\end{lem}

\begin{proof}
Let $H_S(p)=\set{p^m: S(\xi_{p^m})=0, 1\leq m\leq n}$, and similarly
define $H_{\La}(p)$, for a spectrum $\La\subseteq\Z_N$. Suppose that
\[H_{\La}(p)=\set{p^{m_1},\dotsc,p^{m_k}},\]
where $1\leq m_1<m_2<\dotsb<m_k\leq n$. For every $j$, it holds
\begin{equation}\label{ajdiff}
S_{j\bmod q^2}-S_{j\bmod q^2}\subseteq (S-S)\cap
q^2\Z_N\subseteq\set{0}\cup \bigcup_{i=0}^k
\frac{N}{p^{m_i}}\Z_N^{\star},
\end{equation}
by \eqref{divclass}. Consider the $p$-adic expansion of every $s\in
S$ taken $\bmod p^n$, as follows
\[s\equiv s_0+s_1p+\dotsb+s_{n-1}p^{n-1}\bmod p^n,\ 0\leq s_i\leq p-1,\ 0\leq i\leq n-1.\]
Due to \eqref{ajdiff}, the elements of each $S_{j\bmod q^2}$ cannot
have the same $p$-adic digits corresponding to $p^{n-m_i}$, $1\leq
i\leq k$, yielding
\[\abs{S_{j\bmod q^2}}\leq p^k,\    0\leq j<q^2,\]
thus, $\abs{S}\leq p^kq^2$. On the other hand, we have
\[\prod_{i=1}^k\Phi_{p^{m_i}}(x)\mid \La(x),\]
and putting $x=1$ we obtain $p^k\mid \abs{\La}$; we then get by
hypothesis $p^kq^2\mid \abs{S}$, whence $\abs{S}=p^kq^2$, and
\[\abs{S_{j\bmod q^2}}= p^k,\   0\leq j<q^2.\]
Since $S$ is equidistributed $\bmod q^2$, we must also have
$S(\xi_q)=S(\xi_{q^2})=0$ by Lemma \ref{lemegyenessav}. We note that
each element of $S_{j\bmod q^2}$ is unique $\bmod p^n$, so the
reduction $\bmod p^n$ map
\[\pi:\Z_N\mapsto \Z_{p^n}\]
is injective on each $S_{j\bmod p^n}$; fix some $j$, and let
$\pi(S_{j\bmod p^n})=S'$. Since $q^2\mid s-s'$ for every $s,s'\in
S_{j\bmod q^2}$, we conclude that the order of $s-s'$ in $\Z_N$ is
the same as the order of $\pi(s-s')$ in $\Z_{p^n}$, which gives
\[S'-S'\ssq\set{0}\cup\bigcup_{i=0}^k p^{n-m_i}\Z_{p^n}^{\star}.\]
Consider now the set $\La'\ssq\Z_{p^n}$ whose mask polynomial is
given by
\[\La'(x)\equiv \prod_{i=1}^k\Phi_{p^{m_i}}(x)\bmod(x^{p^n}-1).\]
We have $\abs{S'}=\abs{\La'}=p^k$ and
\[S'-S'\ssq\set{0}\cup\set{d\in\Z_{p^n}:\La'(\xi_{p^n}^d)=0},\]
therefore, $(S',\La')$ is a spectral pair in $\Z_{p^n}$ by
\eqref{divclass}. Since
\[\Phi_{p^{m_i}}(x)=1+x^{p^{m_i-1}}+x^{2p^{m_i-1}}+ \ldots + x^{(p-1)p^{m_i-1}}\]
we obtain
\[(\La'-\La')\cap p^{n-m_i+1}\Z_{p^n}^{\star}\neq\vn,\  1\leq i\leq k,\]
therefore,
\[\bigcup_{i=0}^k p^{n-m_i+1}\Z_{p^n}^{\star}\ssq \set{d\in \Z_{p^n}:S'(\xi_{p^n}^d)=0},\]
by \eqref{divclass}, or equivalently
\[\prod_{i=0}^k\Phi_{n-m_i+1}(x)\mid S_{j\bmod q^2}(x),\]
since
\[S_{j\bmod q^2}(x)\equiv S'(x)\bmod(x^{p^n}-1).\]
Moreover, by $S(x)=\sum_{j=0}^{q^2-1}S_{j\bmod q^2}(x)$ and
$\abs{S}=p^kq^2$, we conclude that
\[H_S=\set{p^{n-m_k+1},\dotsc,p^{n-m_1+1},q,q^2},\]
hence $S$ satisfies \ref{t1}.

Consider next the polynomial $F(X)$ satisfying
\[S_{j\bmod q^2}(x)\equiv x^jF(x^{q^2})\bmod(x^N-1),\]
for a fixed $j$. Since $\Phi_{p^{n-m_i+1}}(x)\mid F(x^{q^2})$ for
all $1\leq i \leq k$ and $q^2$ is prime to $p^{n-m_i+1}$, we also
get that $\Phi_{p^{n-m_i+1}}(x)\mid F(x)$. Therefore, for $\ell=1$
or $2$ we get
\[S_{j\bmod q^2}(\xi_{p^{n-m_i+1}q^{\ell}})=\xi_{p^{n-m_i+1}q^{\ell}}^jF(\xi_{p^{n-m_i+1}q^{\ell}}^{q^2})=\xi_{p^{n-m_i+1}q^{\ell}}^jF(\xi_{p^{n-m_i+1}}^{q^{2-\ell}})=0,\]
for all $j$, which shows that $S$ satisfies \ref{t2}. This completes
the proof.
\end{proof}

We distinguish now the following cases:

\vspace{3mm} \noindent $\boxed{S(\xi^q_N)=S(\xi_N^{q^2})=0}$ Then,
since $S(\xi_N)=0$ by Corollary \ref{Nroot}, $S$ is a union of
$\Z_p$-cosets by Lemma \ref{ma} and $S$ tiles due to Reduction
1\ref{enumsummaryb}.

\vspace{3mm} \noindent $\boxed{S(\xi_N^q)S(\xi_N^{q^2})\neq0}$
Consider the difference sets $\La_{j\bmod q}-\La_{j\bmod q}$. They
are always subsets of $(\La-\La)\cap q\Z_N$, but since they avoid
$q\Z_N^{\star}\cup q^2\Z_N^{\star}$ in this case by
\eqref{divclass}, we get
\[\La_{j\bmod q}-\La_{j\bmod q}\subseteq pq\Z_N,\]
for all $j$. This shows that every element of $\La_{j \bmod q}$ has
the same remainder $\bmod p$, or equivalently, for every $j$ there
is an $i=i(j)$ such that
\[\La_{j\bmod q}\subseteq \La_{i(j)\bmod p}.\]
This, in particular, shows that $p<q$, and that every $\La_{i\bmod
p}$ is the disjoint union of sets of the form $\La_{j\bmod q}$,
namely
\[\La_{i\bmod p}=\bigcup_{i(j)=i}\La_{j\bmod q}.\]
Suppose that the number of sets appearing in the union are $\ell$.
Then, the above equation along with \eqref{ed} implies $1/p=\ell/q$,
which leads to a contradiction (no such spectrum can exist).

\vspace{3mm} \noindent $\boxed{S(\xi_N^q)=0\neq S(\xi_N^{q^2})}$ We
apply Theorem \ref{vansumspos} to $S(x)\bmod (x^{N/q}-1)$. We obtain
\[S(x)\equiv P(x)\Phi_p(x^{N/pq})+Q(x)\Phi_q(x^{N/q^2})\bmod(x^{N/q}-1),\]
since $S(\xi_{N/q})=0$, where $P(x)$ and $Q(x)$ have nonnegative
coefficients. Furthermore, since $S(\xi_N^{q^2})\neq0$, we cannot
have $Q\equiv 0$. Due to the nonnegativity of $P$ and $Q$, we obtain
the existence of $s,s'\in S$ such that
\[s-s'\equiv\frac{N}{q^2}\bmod\frac{N}{q},\]
hence $p^n\mid s-s'$ but $q\nmid s-s'$, yielding $s-s'\in
p^n\Z_N^{\star}$ and
\[\La(\xi_{q^2})=0,\]
which further gives $q^2\mid \abs{\La}$, so by Lemma
\ref{q2divides}, $S$ tiles $\Z_N$.

\vspace{3mm} \noindent $\boxed{S(\xi_N^q)\neq0=S(\xi_N^{q^2})}$ We
will prove the following:
\begin{clm}
$(S-S)\cap\frac{N}{pq^2}\Z_N^{\star}\neq\vn$.
\end{clm}
\begin{proof}[Proof of Claim]
By Theorem \ref{vansumspos}, the multiset\footnote{Here, we consider
the elements $q^2 s\bmod N$, $s\in S$, counting multiplicities. For
example, if $N=4$ and $S=\set{0,2}$, then $2S$ is the multiset whose
only element is $0$, appearing with multiplicity $2$.} $q^2 S$ is a
union of $\Z_p$-cosets, or equivalently
\begin{equation}\label{eqclaim11}
\abs{S_{i\bmod p^n}}=\abs{S_{i+kp^{n-1}\bmod p^n}},\end{equation}
for every $i,k$. We partition the above sets $\bmod p^nq$:
\[S_{i\bmod p^n}=\bigcup_{\ell=0}^{q-1}S_{i+\ell p^n\bmod p^nq},\]
and
\[S_{i+kp^{n-1}q\bmod p^n}=\bigcup_{\ell=0}^{q-1}S_{i+kp^{n-1}q+\ell p^n\bmod p^nq}.\]
If for every $i$ existed some $\ell$ such that
\[S_{i+kp^{n-1}q\bmod p^n}=S_{i+kp^{n-1}q+\ell p^n\bmod p^nq},\]
for every $k$, then $qS$ would also be a union of $\Z_p$-cosets.
Indeed, as for every $i$ there is at most one value of
$0\leq\ell\leq q-1$ such that $S_{i+\ell p^n\bmod p^nq}\neq\vn$, and
by the above condition the cardinalities of $S_{i+kp^{n-1}q+\ell
p^n\bmod p^nq}$ are the same for $0\leq k\leq p-1$. Therefore,
$S(\xi_{p^nq})=0$ by Theorem \ref{vansumspos} (or equivalently by
Proposition \ref{propsqfree}),
contradicting the hypothesis. Thus, there exists $i$ such that there
are nonempty $S_{i+\ell p^n\bmod p^nq}$ and $S_{i+\ell' p^n\bmod
p^nq}$, with $0\leq \ell<\ell'\leq q-1$. Clearly, $S_{i+ \ell
p^n\bmod p^nq} \subseteq S_{i \bmod p^n}$, so $S_{i \bmod p^n}$ is
nonempty. Using \eqref{eqclaim11} we have $S_{i+p^{n-1}\bmod p^n}$
is nonempty.

Now let $s\in S_{i+p^{n-1}\bmod p^n}$, $s'\in S_{i+\ell p^n\bmod
p^nq}$ and $s''\in S_{i+\ell' p^n\bmod p^nq}$, so that $p^{n-1}\pdiv
s-s'$ and $p^{n-1}\pdiv s-s''$. Since
$s''-s'\equiv(\ell'-\ell)p^n\bmod p^nq$, we get $q\nmid s''-s'$, so
either $q\nmid s-s'$ or $q\nmid s-s''$ would hold, yielding
$(S-S)\cap p^{n-1}\Z_N^{\star}\neq\vn$, as desired.
\end{proof}
This implies
\begin{equation}\label{Brootpq2}
\La(\xi_{pq^2})=0,
\end{equation}
by \eqref{ordroots}. If $\La(\xi_{q^2})=0$ then we would have
$q^2\mid \abs{S}$ and $S$ would tile $\Z_N$ by virtue of Lemma
\ref{q2divides}. So, we may assume $\La(\xi_{q^2})\neq0$.

By \eqref{Brootpq2} and Theorem \ref{vansumspos} we get
\[\La(x)\equiv\sum_{j=0}^{pq^2-1}\abs{\La_{j\bmod pq^2}}x^j\equiv P(x)\Php+Q(x)\Phq\bmod(x^{pq^2}-1),\]
for some $P(x),Q(x)\in\Z_{\geq0}[x]$ and $P(x)\not\equiv0$ by
$\La(\xi_{q^2})\neq0$. We note that the function
$f(j)=\abs{\La_{j\bmod pq^2}}$ restricted on a $\Z_{pq}$-coset of
$\Z_{pq^2}$ is supported either on a $\Z_p$-coset or a $\Z_q$-coset;
otherwise, there would exist $\la\in \La_{j\bmod pq^2}$ and
$\la'\in \La_{j'\bmod pq^2}$ where $j,j'$ satisfy
\[j-j'\in \out{q}\Z_{pq^2}^{\star}.\]
This shows that $\out{q}\pdiv \la-\la'$ and $\out{p}\nmid \la-\la'$,
thus $\la-\la'\in \out{q}\Z_N^{\star}$
 and $S(\xi_N^{\out{q}})=0$ by \eqref{divclass}, contradicting the hypothesis.

Next, consider a nonempty subset $\La_{j\bmod pq^2}$; the
polynomials with nonnegative coefficients $P(x)\Php$ and $Q(x)\Phq$
contribute to the coefficient of $x^j$ of $\La(x)\bmod(x^{pq^2}-1)$.
If both contributions are positive, then all subsets
$\La_{j+kq^2\bmod pq^2}$ and $\La_{j+\ell pq\bmod pq^2}$ are
nonempty, for $0<k<p$ and $0<\ell<q$. Then, for $\la\in
\La_{j+q^2\bmod pq^2}$ and $\la'\in \La_{j+pq\bmod pq^2}$, we have
$q\pdiv \la-\la'$, hence $\la-\la'\in q\Z_N^{\star}$, which
contradicts $S(\xi_N^q)\neq0$, due to \eqref{ordroots}.

Let $\Gamma(x)$ be $\La(x)\bmod(x^{pq^2}-1)$. The previous argument
shows that the coefficient of $x^j$ of $\Gamma(x)$ is determined
completely either from $P(x)\Php$ or $Q(x)\Phq$. Moreover, if $q
\pdiv j-j'$, then we cannot have that both the coefficients of $x^j$
and $x^{j'}$ in $\Gamma(x)$ are nonzero by the same argument. This
means that $f(j)=\abs{\La_{j\bmod pq^2}}$ restricted on a
$\Z_{pq}$-coset of $\Z_{pq^2}$ is supported either on a $\Z_p$-coset
or a $\Z_q$-coset and constant restricted to this coset.

This shows that for each $j$ such that $\La_{j\bmod pq^2}\neq\vn$,
either
\begin{equation}\label{absorption}
\abs{\La_{j+kq^2\bmod pq^2}}=\frac{1}{p}\abs{\La_{j\bmod
q^2}}=\frac{1}{p}\abs{\La_{j\bmod q}},\    0\leq k<p,
\end{equation}
or
\begin{equation}\label{equidistribution}
 \abs{\La_{j+\ell pq\bmod pq^2}}=\abs{\La_{j+\ell pq\bmod q^2}}=\frac{1}{q}\abs{\La_{j\bmod q}},\   0\leq \ell<q
\end{equation}
holds. If \eqref{equidistribution} holds for some $j$, then $q^2\mid
\abs{\La}$, so by Lemma \ref{q2divides} we get that $S$ tiles
$\Z_N$. Therefore, we may assume that \eqref{absorption} holds for
all $j$ with $\La_{j\bmod pq^2}\neq\vn$. For such $j$, we have
\[\La_{j\bmod q}-\La_{j\bmod q}=\La_{j\bmod q^2}-\La_{j\bmod q^2}\subseteq (\La-\La)\cap q^2\Z_N,\]
hence $\La-\La$ completely avoids $q\Z_N\setminus q^2\Z_N$. On the
other hand, $(S-S)\cap p^n\Z_N^{\star}=\vn$ by \eqref{divclass} and
the assumption $\La(\xi_{q^2})\neq0$, hence the polynomials
\[\ol{S}(x)\equiv S(x)\Phi_q(x^{p^n})\bmod(x^N-1)\]
and
\[\ol{\La}(x)\equiv \La(x)\Phq\bmod(x^N-1)\]
 are mask polynomials of subsets of $\Z_N$, say $\ol{S}$ and
 $\ol{\La}$, i.e. their coefficients are $0$ or $1$. We
 claim that $(\ol{S},\ol{\La})$ is a spectral pair. They
 obviously have the same cardinality of $q |S|$, and an element of $\ol{\La}-\ol{\La}$ can be expressed as $\la-\la'+lN/q$, where $\la,\la'\in \La$, $\abs{l}<q$.

If $\la-\la'\in p^k\Z_N^{\star}$, then $q\nmid \la-\la'+lN/q$, hence
$\la-\la' +lN/q \in p^k\Z_N^{\star}$ as well, yielding
$S(\xi_N^{\la-\la'+lN/q})=0$, since $N/q=p^nq$.

The remaining case is $\la-\la'\in p^kq^2\Z_N^{\star}$,
 where $0\leq k\leq n-1$, as $\La-\La$ avoids $q\Z_N\setminus
 q^2\Z_N$. In this case, $\la-\la'+lN/p\in
 p^kq\Z_N^{\star}$ if $1 \le l \le q-1$, and
 $\Phi_q(\xi_{p^kq}^{p^n})=\Phi_q(\xi_q^{p^{n-k}})=0$, so
 \[\ol{S}(\xi_N^{\la-\la'+lN/p})=0.\]
 If $l=0$, then clearly $\la-\la' \in \lala$.
 Considering all of these cases we have $\ol{\La}-\ol{\La}\subseteq\set{0}\cup Z(\ol{S})$, proving that the pair $(\ol{S},\ol{\La})$ is spectral by virtue of \eqref{divclass}.
 Since $q^2\mid\ol{S}$ we have $\ol{S}$ tiles $\Z_N$ by Lemma \ref{q2divides}, thus there is $T\subseteq\Z_N$ such that
 \[S(x)\Phi_q(x^{p^n})T(x)\equiv \ol{S}(x)T(x)\equiv 1+x+\dotsb+x^{N-1}\bmod (x^N-1),\]
 so $\Phi_q(x^{p^n})T(x)$ is the mask polynomial of a tiling complement of $S$ using Lemma 1.3 in \cite{CM1999}, completing the proof. \qed




\section{Appendix}
\begin{Thm}\label{ThmZp2}
Let $S$ be a subset of $\Z_p^2$. Then $S$ tiles $\Z_p^2$ if and only
if $S$ is spectral.
\end{Thm}

Iosevich et al. \cite{IMP2017} has already proved this theorem, but
we provide an easy combinatorial proof for one of the two directions
and a short one for the other direction using R\'edei's theorem.

\begin{prop}\label{thm91}
Let $S$ be a spectral set of $\Z_p^2$. Then $S$ tiles $\Z_p^2$.
\end{prop}
\begin{proof}
Let $S$ be a spectral set. We might assume $|S|>1$, since one
element sets clearly tile every group. The corresponding spectrum
$\Lambda$ is also of size at least $2$. Then there is a nontrivial
irreducible representation $\psi$ of $\Z_p^2$ such that $\sum_{s \in
S} \psi(s)=0$. We may also assume $|S|=|\La|<p^2$.

Representations of $\mathbb{Z}_p^2$ can be parametrized by the
elements of $\mathbb{Z}_p^2$. For $u \in \mathbb{Z}_p^2$ let
$\chi_u(v)=e^{\frac{2 \pi i \langle u,v \rangle}{p}}$, where the
scalar product of $\langle u,v \rangle$ is taken modulo $p$. This
can be written as $\sum_{j=0}^{p-1} a_je^{\frac{2 \pi i}{p}j}$,
where $a_j$'s are integers, which are determined in the following
way.

From now on we may also think of $\Z_p^2$ as a $2$ dimensional
vector space over $\Z_p$. Cosets of $1$ dimensional subspaces are
called lines. Let $u'$ be a nonzero vector orthogonal to $u$. Then
$\langle u,v \rangle$ is constant on every coset of the subgroup
generated by $u'$. Basically, we count the intersection of $S$ with
the elements the set of lines parallel with $\langle u' \rangle$.
The expression $\sum_{j=0}^{p-1} a_je^{\frac{2 \pi i }{p}j}=0$ if
and only if $a_j$ is a constant sequence so every element of this
class of parallel lines contains the same amount of elements of $S$
(i.e. $S$ is equidistributed on these set of parallel lines). As a
consequence we get that $p \mid |S|$.

If $|S|=p$, then by the previous argument we have that $S$
intersects each element of a class of parallel lines once. Then $S$
clearly tiles $\Z_p^2$.

Thus we may assume $p+1<2p\le |\Lambda|=|S|<p^2 $. It is enough to
show that such spectral set does not exist. Each class of parallel
lines consists of $p$ lines. Thus we have that for every class of
parallel lines, at least one line contains at least two elements of
$\Lambda$. Thus using the argument above we have that every element
of every class of parallel lines contains the same amount of
elements of $S$. This means that every line contains  $k$ elements
of $S$ for some fixed number $p>k \ge 2$.

Let $x \in \Z_p^2 \setminus S$. Take every line containing $x$.
These lines give a disjoint cover of $\Z_p^2 \setminus \{ x\}$.
Since each of them contains $k\ge 2$ elements we have $|S|=(p+1)k$,
which is not divisible by $p$, a contradiction.
\end{proof}
\begin{prop}
Let $S$ be a set in $\Z_p^2$, which tiles. Then $S$ is spectral.
\end{prop}
\begin{proof}
We may assume that $1<|S|<p^2$. Then $S$ and its tiling complements
are of cardinality $p$. Using Theorem \ref{redei}
 we obtain that either $S$ or $T$ is a subgroup of $\Z_p^2$.

Subgroups are clearly spectral sets. If $T$ is a subgroup, then $S$
is a complete set of coset representatives of $T$. Let $U$ denote
the subgroup of $\Z_p^2$ consisting of vectors orthogonal to $T$.
Then clearly $S$ is equidistributed on the orthogonal lines for $u
\in U$ so $U$ is a spectrum for $S$.
\end{proof}

\section*{Acknowledgements}
The authors wish to thank the anonymous referees, as well as the editor, Professor Izabella \L{}aba, for improving the quality of our paper with their suggestions.

\noindent Research of G. Kiss was supported by the Hungarian National
Foundation for Scientific Research, Grant No. K 124749 and the
internal research project R-AGR-0500.

\noindent Research of G. Somlai was supported by the Hungarian
National Foundation for Scientific Research, Grant No. K 115799.

\noindent Research of M. Vizer was supported by the Hungarian
National Foundation for Scientific Research, Grant No. SNN 116095,
129364, KH 130371 and K 116769.


\begin{thebibliography}{99}
\bibitem{B1979} L. Babai. Spectra of Cayley graphs. \textit{Journal of Combinatorial Theory, Series B}, \textbf{27}(2), 180--189, 1979.

\bibitem{BST1973} L. Brown, B. Schreiber, B. A. Taylor. Spectral synthesis and the Pompeiu problem, {\it Annales de l'institut Fourier} {\bf 23}(3), 125--154, 1973.


\bibitem{C1944} L. Chakalov. Sur un probleme de D. Pompeiu. \textit{Godishnik Univ. Sofia, Fac. Phys.-Math.},
Livre 1, \textbf{40}, 1--14, 1944.

\bibitem{CM1999} E. M. Coven, A. Meyerowitz. Tiling the integers with translates of one finite set. \textit{Journal of Algebra}, \textbf{212}(1), 161--174, 1999.

\bibitem{DL2014} D. E. Dutkay, C. K. Lai. Some reductions of the spectral set conjecture to integers. \textit{Mathematical Proceedings of the Cambridge Philosophical Society}, \textbf{156}(1), 123--135, 2014.


\bibitem{FFLS2015} A. Fan, S. Fan, L. Liao, R. Shi. Fuglede's conjecture holds in $\mathbb{Q}_p$. \textit{ArXiv preprint}, 2015: \url{https://arxiv.org/abs/1512.08904}

\bibitem{FMM2006} B. Farkas, M. Matolcsi, P. M\'ora. On Fuglede's conjecture and the existence of universal spectra. \textit{Journal of Fourier Analysis and Applications}, \textbf{12}(5), 483--494, 2006.

\bibitem{F1974} B. Fuglede. Commuting self-adjoint partial      differential operators and a group theoretic problem. \textit{Journal of Functional Analysis}, \textbf{16}(1), 101--121, 1974.
\bibitem{G1988} N. Garofalo. A new result on the Pompeiu problem. \textit{Rend. Sem. Mat. Univ. Pol. Torino} P.D.E and Geometry, 25--38, 1988.

\bibitem{IKP99} Iosevich, N. Katz, S. Pedersen. Fourier bases and a distance problem of Erd\H{o}s. \textit{Math. Res. Lett.},  \textbf{6} (2), 251--255, 1999.

\bibitem{IKT2003} A. Iosevich, N. Katz, T. Tao. The Fuglede spectral conjecture holds for convex planar domains. \textit{Mathematical Research Letters}, \textbf{10}(5), 559--569, 2003.

\bibitem{IMP2017} A. Iosevich, A. Mayeli, J. Pakianathan. The Fuglede conjecture holds in $\mathbb{Z}_p \times \mathbb{Z}_p$. \textit{Analysis \& PDE}, \textbf{10}(4), 757--764, 2017.

\bibitem{KLV2016} G. Kiss, M. Laczkovich, Cs. Vincze. The discrete Pompeiu problem on the plane. \textit{Monatshefte f\"ur Mathematik}, \textbf{186}(2), 299--314, 2018.


\bibitem{KM2006} M. N. Kolountzakis, M. Matolcsi. Complex Hadamard matrices and the spectral set conjecture. \textit{Collectanea Mathematica}, Vol. Extra, 281--291, 2006.

\bibitem{KM2006a} M. N. Kolountzakis, M. Matolcsi. Tiles with no spectra. \textit{Forum Mathematicum} \textbf{18}(3), 519--528, 2006.

\bibitem{L2001} I. \L aba. Fuglede's conjecture for a union of two intervals. \textit{Proceedings of the American Mathematical Society}, \textbf{129}(10), 2965--2972, 2001.

\bibitem{L2002} I. \L aba. The spectral set conjecture and multiplicative properties of roots of polynomials. \textit{Journal of the London Mathematical Society}, \textbf{65}(3), 661--671, 2002.

\bibitem{LS2004} M. Laczkovich, G. Sz\'ekelyhidi. Harmonic analysis on discrete abelien group. {\it Proc. Am. Math. Soc.} {\bf 133}(6), 1581--1586, 2004.

\bibitem{vanishingsum} T. Y. Lam, K. H. Leung. On Vanishing Sums of Roots of Unity. \textit{Journal of Algebra} \textbf{224}, 91--109, 2000.



\bibitem{Ma}
S.~L. Ma. Polynomial addition sets. \textit{Ph.D. thesis},
University of Hong
  Kong, 1985.

\bibitem{MK17} R. D. Malikiosis, M. N. Kolountzakis. Fuglede's conjecture on cyclic groups of order $p^nq$. \textit{Discrete Analysis}, 2017:12, 16pp.

\bibitem{M2004} M. Matolcsi. Fuglede's conjecture fails in dimension 4. \textit{Proceedings of the American Mathematical Society}, \textbf{133}(10), 3021--3026, 2005.



\bibitem{Newman1997} D.J. Newman. Tesselations of integers. \textit{J. Number Theory} \textbf{9}, 107--111, 1997.


\bibitem{P1929} D. Pompeiu. Sur une propri\'et\'e int\'egrale des fonctions de deux variables r\'eelles. \textit{Bull. Sci. Acad. Roy. Belgique}, \textbf{5}(15), 265--269, 1929.

\bibitem{Pott}
Alexander Pott. Finite geometry and character theory.
\textit{Lecture Notes in Mathematics}, vol. 1601, Springer-Verlag,
Berlin, 1995.

\bibitem{Pu2013} M. J. Puls. The Pompeiu problem and discrete groups. \textit{Monatshefte f\"ur Mathematik}, \textbf{172}(3-4), 415--429, 2013.

\bibitem{R1997} A. G. Ramm. The Pompeiu problem. \textit{Applicable Analysis}, \textbf{64}(1-2), 19--26, 1997.

\bibitem{Ram2017} A. G. Ramm. Solution to the Pompeiu problem and the related symmetry problem. {\it Applied Mathematics Letters}, {\bf 63},  28--33, 2017.

\bibitem{Redei65}
L. R\'edei. Die neue Theorie der endlichen abelschen Gruppen und
Verallgemeinerung des Hauptsatzes von Haj\'os. \textit{Acta Math.
Acad. Sci. Hung.}, \textbf{16}, 329--373, 1965.

\bibitem{SchmidtFDM}
B.~Schmidt. Characters and cyclotomic fields in finite geometry.
\textit{Lecture Notes in Mathematics}, vol. 1797, Springer-Verlag,
Berlin, 2002.

\bibitem{Shi18}
R. Shi. Fuglede's conjecture holds on cyclic groups $\Z_{pqr}$.
\textit{ArXiv preprint}, 2018:
\url{https://arxiv.org/abs/1805.11261}

\bibitem{reprsteinberg} B. Steinberg. Representation theory of finite groups: An introductory approach. \textit{Springer Science and Business media}, 2012.

\bibitem{Sz2001}
L. Sz\'ekelyhidi. A Wiener Tauberian theorem on discrete Abelian
torsion groups. {\it Ann. Acad. Paedag. Cracov. Studia Math. I},
{\bf 4}, 147--150, 2001.

\bibitem{T2004} T. Tao. Fuglede's conjecture is false in 5 and higher dimensions. \textit{Mathematical Research Letters}, \textbf{11}(2), 251--258, 2004.

\bibitem{Z1992} L. Zalcman. A bibliographic survey of the Pompeiu problem. \textit{Approximation by solutions of partial differential equations}, Kluwer
Acad., Dordrecht. \textbf{365}, 185--194, 1992.

\bibitem{Ze1978} D. Zeilberger. Pompeiu's problem in discrete space. {\it Proc. Nat. Acad. Sci. USA} {\bf 75}(8), 3555--3556, 1978.




\end{thebibliography}
\end{document}